\documentclass[a4paper,reqno,10pt]{amsart}

\usepackage{epsf,graphicx,epsfig}
\usepackage{amscd}
\usepackage{cancel}
\usepackage{amsmath,latexsym,amssymb,amsthm}
\usepackage[nospace,noadjust]{cite}
\usepackage{textcomp}
\usepackage{setspace,cite}
\usepackage{lscape,fancyhdr,fancybox}
\usepackage{stmaryrd}
\usepackage[all,cmtip]{xy}
\usepackage{tikz}
\usepackage{cancel}
\usetikzlibrary{shapes,arrows,decorations.markings}
\usepackage{graphicx}

\usepackage{color}
\usepackage{url}
\usepackage{enumerate}
\usepackage[mathscr]{euscript}

  \theoremstyle{plain}

\swapnumbers
    \newtheorem{thm}{Theorem}[section]
    \newtheorem{prop}[thm]{Proposition}
     
   \newtheorem{lemma}[thm]{Lemma}
    \newtheorem{corollary}[thm]{Corollary}
    
    \newtheorem{subsec}[thm]{}
\theoremstyle{definition}
    \newtheorem{defn}[thm]{Definition}
        \newtheorem{remark}[thm]{Remark}
    \newtheorem{exam}[thm]{Example}

\theoremstyle{remark}

\title{}
\author{}
\date{}
\usepackage{amssymb}

\usepackage{hyperref}
\hypersetup{
	colorlinks,
	citecolor=blue,
	filecolor=black,
	linkcolor=blue,
	urlcolor=black
}

\begin{document}
\title{Extensions and deformations of algebras with higher derivations}

\author{Apurba Das}
\address{Department of Mathematics and Statistics,
Indian Institute of Technology, Kanpur 208016, Uttar Pradesh, India.}
\email{apurbadas348@gmail.com}

\curraddr{}
\email{}

\subjclass[2010]{}

\begin{abstract}
Higher derivations on an associative algebra generalizes higher order derivatives. We call a tuple consisting of an algebra and a higher derivation on it by an AssHDer pair. We define a cohomology for AssHDer pairs with coefficients in a representation. Next, we study central extensions, abelian extensions of AssHDer pairs and relate them to the second cohomology group with appropriate coefficients. Deformations of AssHDer pairs are also considered which are governed by the cohomology with self coefficient.
\end{abstract}

\noindent

\thispagestyle{empty}

\maketitle

\tableofcontents


\section{Introduction}
Let $A$ be an associative algebra. A linear map $d : A \rightarrow A$ is called a derivation on $A$ if $d$ satisfies $d(ab) = d(a) b + a d(b)$, for all $a, b \in A$. Derivations are generalization of usual derivatives. They are important object to study algebras. They are also useful to construct deformation formulas \cite{coll}, homotopy Lie algebras \cite{voro} and they have applications in control theory, gauge theories in quantum field theory \cite{ ref-1,ref-2 } and differential Galois theory \cite{magid}. Algebras with derivations are studied from operadic point of view in \cite{loday}. Recently, Lie algebras with derivations (called LieDer pairs) are studied from cohomological perspectives and extension, deformation of LieDer pairs are considered \cite{tang-f-sheng}. These results has been extended to associative algebras and Leibniz algebras with derivations (AssDer pairs and LeibDer pairs) in \cite{das-mandal, das-leibder}.

If $d$ is a derivation on $A$, then the sequence of maps $\{ d_1, \ldots, d_N \}$ on $A$ defined by $d_k = \frac{d^k}{k!}$ satisfy the following relations
\begin{align}\label{eq-first}
d_k (ab) = \sum_{i+j = k} d_i(a) d_j (b),
\end{align}
for $k = 1, \ldots, N$, with the convention that $d_0 = \mathrm{id}_A$. This motivates the notion of higher derivations. More precisely, a finite sequence $\{ d_1, \ldots, d_N \}$ of linear maps on $A$ is said to be a higher derivation of rank $N$ on $A$ if they satisfy the identities (\ref{eq-first}). Higher derivations were introduced by Hasse and Schmidt \cite{hasse}, therefore, they are also called {\em Hasse-Schmidt} derivations. Not all higher derivations arise from derivations in the above way. In \cite{abu, mirz} the authors proved that each higher derivation on $A$ is a combination of compositions of derivations, hence, characterize all higher derivations in terms of derivations on $A$. Higher derivations on triangular algebras are considered in \cite{wei}. See \cite{hers , csl , jordan} for more results about higher derivations. Higher derivations are useful in theory of automorphisms of complete local rings. In fact, higher derivations can be applied to the Galois theory of fields, to universal higher derivations and to separability criteria \cite{davis, weis}.

An algebra $A$ together with a distinguished higher derivation $\{ d_1, \ldots, d_N \}$ is called an AssHDer pair. In this paper, we study AssHDer pairs from cohomological perspectives. We define representations of AssHDer pairs and construct the semi-direct product. Higher derivations on Lie algebras are also studied. The standard commutator gives rise to a functor from the category of AssHDer pairs to the category of LieHDer pairs. The universal enveloping algebra construction gives rise to a functor left adjoint to the above functor. We define cohomology of AssHDer pairs with coefficients in a representation. This is a generalization of the cohomology of AssDer pairs introduced in \cite{das-mandal}.

Next we study central extension of AssHDer pairs and show that isomorphism classes of central extensions are classified by the second cohomology group of AssHDer pair with trivial representation (Theorem \ref{cent-thm}). We also consider abelian extensions of an AssHDer pair and show that they are classified by the second cohomology with arbitrary representation (Theorem \ref{abel-cla}). The classical deformation theory of Gerstenhaber \cite{gers} has been extended to AssDer pairs in \cite{das-mandal}. Here we further extend this approach to AssHDer pairs. The vanishing of the second cohomology of an AssHDer pair implies that the structure is rigid (Theorem \ref{2-thm}). Given a finite order deformation, we associate a third cocycle, called the obstruction cocycle (Proposition \ref{3-thm}). When this cocycle is a coboundary, the deformation extends to a deformation of next order. As a consequence, the vanishing of the third cohomology implies that any finite order deformation is always extensible (Theorem \ref{final-thm}).

We end this paper with some further discussions. All vector spaces and linear maps in this paper are over a field $\mathbb{K}$ of characteristic zero.

\section{Higher derivations and AssHDer pairs}
In this section, we recall higher derivations and construct free higher derivation on a vector space with linear maps. We also define higher derivations on Lie algebras and give some functorial constructions.

Let $A$ be an associative algebra.

\begin{defn}
A {\em higher derivation (of rank $N$)} on $A$ is a finite sequence $\{ d_1, \ldots, d_N \}$ of linear maps on $A$ satisfying the following identities
\begin{align}\label{hder-id}
d_k (ab) = \sum_{i+j = k} d_i (a) d_j (b),
\end{align}
for $ k = 1, \ldots, N$ and $a, b \in A,$ with the convention that $d_0 = \mathrm{id}_A$.

It follows from (\ref{hder-id}) that $d_1$ is a derivation on $A$. In particular, if $d$ is a derivation on $A$ then $\{ d \} $ is a higher derivation of rank $1$. In fact, if $d$ is a derivation on $A$ then $\{ d_1 = 0, \ldots, d_i = d, \ldots, d_N = 0 \}$ is a higher derivation of rank $N$.
\end{defn}

\begin{remark}\label{higher-der-as-mor}
Note that, if $A$ is an associative algebra, then $A[[t]]/ (t^{N+1})$ has an associative algebra structure with the product induced from that of $A$. Then a higher derivation on $A$ is same as a morphism of algebras
\begin{align*}
 A \rightarrow A[[t]]/ (t^{N+1}), ~~ a \mapsto a + d_1 (a) t + \cdots + d_N (a) t^N.
\end{align*}
\end{remark}

An associative algebra $A$ together with a higher derivation of rank $N$ is called an AssHDer pair of rank $N$. An AssHDer pair as above may be denoted by the tuple $(A, d_1, \ldots, d_N)$.

\begin{exam}
\begin{itemize}
\item Let $d$ be a derivation on $A$. Then $\{ d, \frac{d^2}{2 !} , \ldots, \frac{d^N}{N !} \}$ is a higher derivation of rank $N$. Such higher derivations are called ordinary higher derivations. Thus, in the polynomial algebra $A = \mathbb{K}[x_1, \ldots, x_n ]$, the divided power operators $\{ \frac{\partial}{\partial x_i}, \frac{1}{2!} \frac{\partial^2}{\partial^2 x_i}, \ldots, \frac{1}{N!} \frac{\partial^N}{\partial^N x_i} \}$ forms a higher derivation of rank $N$ on $A$.
This example shows that higher derivations are generalization of higher derivatives. 
\item Let $A$ be an associative algebra and $x \in A$ be any fixed element. Then the maps
$$d_n (a) := x^{n-1} (xa - ax), ~~~ \text{ for } a \in A \text{ and } 1 \leq n \leq N$$
forms a higher derivation of rank $N$ \cite{now}.
\item Let $\{ d_1, \ldots, d_N \}$ be a higher derivation on $A$. For any $1 \leq q \leq N$, we define a new sequence $\{d_1', \ldots, d_N' \}$ of linear maps by
\begin{align*}
d_k' = \begin{cases} 0 & ~~~ \text{ if } q \nmid k,\\
d_s  & ~~~ \text{ if } k = sq. \end{cases}
\end{align*}
Then $\{d_1', \ldots, d_N'\}$ is a higher derivation of rank $N$ \cite{now, wei}.
\item Let $A$ be an unital associative algebra. Suppose there exist two sequences $\{ x_1, \ldots, x_N \}$ and $\{ y_1, \ldots, y_N \}$ of elements in $A$ satisfying the conditions $\sum_{i=0}^n x_i y_{n-i} = \delta_{n0} = \sum_{i=0}^n y_i x_{n-i}$ (with the convention that $x_0 = y_0 = 1_A$). Then the maps $d_n : A \rightarrow A, ~ 1 \leq n \leq N$, defined by
$$ d_n (a) = \sum_{i=0}^n x_i a y_{n-i}$$
forms a higher derivation of rank $N$. This kind of higher derivations are called inner higher derivations \cite{wei}.
\end{itemize}
\end{exam}

\begin{defn}
Let $(A, d_1, \ldots, d_N )$ and $(A', d_1', \ldots, d_N' )$ be two AssHDer pairs of rank $N$. A morphism between them is given by an algebra morphism $f : A \rightarrow A'$ satisfying
\begin{align*}
d_k' \circ f =  f \circ d_k, ~~~ \text{ for } k = 1, \ldots, N.
\end{align*}
It is called an isomorphism if $f$ is an isomorphism.
\end{defn}

\medskip

Let $V$ be a vector space and $d : V \rightarrow V$ be a linear map. Then $d$ induces a derivation $\overline{d}$ on the tensor algebra $T(V)$ (resp. reduced tensor algebra $(\overline{T}(V)$) by
\begin{align*}
\overline{d} ( v_1 \otimes \cdots \otimes v_n) = \sum_{i=1}^n v_1 \otimes \cdots \otimes dv_i \otimes \cdots \otimes v_n.
\end{align*}
The pair $(T(V), \overline{d})$ is infact a free AssDer pair over $(V, d)$. See \cite{das-mandal} for details. Here we extend this construction for AssHDer pairs.

Let $(V, \vartheta_1, \ldots, \vartheta_N)$ be a vector space together with linear maps. These linear maps induce maps $\overline{\vartheta}_1, \overline{\vartheta}_2, \ldots, \overline{\vartheta}_N$ on the tensor algebra $T(V)$ as
\begin{align*}
&\overline{\vartheta}_1 ( v_1 \otimes \cdots \otimes v_n) = \sum_{i=1}^n v_1 \otimes \cdots \otimes \vartheta_1 v_i \otimes \cdots \otimes v_n,\\
&\overline{\vartheta}_2 (v_1 \otimes \cdots \otimes v_n) = \sum_{i=1}^n v_1 \otimes \cdots \otimes \vartheta_2 v_i \otimes \cdots \otimes v_n + \sum_{i < j} v_1 \otimes \cdots \otimes \vartheta_1 v_i \otimes \cdots \otimes \vartheta_1 v_j \otimes \cdots \otimes v_n,\\
&\overline{\vartheta}_3 (v_1 \otimes \cdots \otimes v_n) = \sum_{i=1}^n v_1 \otimes \cdots \otimes \vartheta_3 v_i \otimes \cdots \otimes v_n + \sum_{i < j} v_1 \otimes \cdots \otimes \vartheta_2 v_i \otimes \cdots \otimes \vartheta_1 v_j \otimes \cdots \otimes v_n \\
& \qquad \qquad \qquad \qquad + \sum_{i < j} v_1 \otimes \cdots \otimes \vartheta_1 v_i \otimes \cdots \otimes \vartheta_2 v_j \otimes \cdots \otimes v_n \\
& \qquad \qquad \qquad \qquad + \sum_{i < j < k} v_1 \otimes \cdots \otimes \vartheta_1 v_i \otimes \cdots \otimes \vartheta_1 v_j \otimes \cdots \otimes \vartheta_1 v_k \otimes \cdots \otimes v_n,\\
&\vdots\\
& \overline{\vartheta}_k (v_1 \otimes \cdots \otimes v_n) = \sum_{s=1}^n \sum_{p_1 + \cdots + p_s = k} \sum_{i_1 < \cdots < i_s} v_1 \otimes \cdots \otimes \vartheta_{p_1} v_{i_1} \otimes \cdots \otimes \vartheta_{p_2} v_{i_2} \otimes \cdots \otimes \vartheta_{p_s} v_{i_s} \otimes \cdots \otimes v_n,\\
&\vdots
\end{align*} 
for $v_1 \otimes \cdots \otimes v_n \in T(V).$ With these notations, we have the following.
\begin{prop}
The tuple $(T(V), \overline{\vartheta}_1, \ldots, \overline{\vartheta}_N)$ is an AssHDer pair of rank $N$. 
\end{prop}
In the following, we show that it is a free AssHDer pair in a certain sense.

Let $(V, \vartheta_1, \ldots, \vartheta_N)$ be a vector space with linear maps. The free AssHDer pair over it is an AssHDer pair $(\mathcal{F}(V), \delta_1, \ldots, \delta_N)$ with a linear map $i : V \rightarrow \mathcal{F}(V)$ satisfying $\delta_k \circ i = i \circ \vartheta_k$, for $k = 1, \ldots, N$ and satisfying the following universal conditions:
\begin{center}
 For any AssHDer pair $(A, d_1, \ldots, d_N)$ and a linear map $f : V \rightarrow A$ satisfying $d_k \circ f = f \circ \vartheta_k$, there exists an unique AssHDer pair morphism $\widetilde{f} : \mathcal{F}(V) \rightarrow A$ satisfying $\widetilde{f} \circ i = f.$
\end{center}

\begin{prop}\label{free-asshd}
Let $(V, \vartheta_1, \ldots, \vartheta_N)$ be a vector space with linear maps. Then $(T(V), \overline{\vartheta}_1, \ldots, \overline{\vartheta}_N)$ together with the inclusion map $i : V \rightarrow T(V)$ is a free AssHDer pair over $(V, \vartheta_1, \ldots, \vartheta_N)$.
\end{prop}

\subsection{Higher derivations on Lie algebras} Let $(\mathfrak{g}, [~, ~])$ be a Lie algebra. A derivation on $\mathfrak{g}$ is given by a linear map $\phi_\mathfrak{g} : \mathfrak{g} \rightarrow \mathfrak{g}$ satisfying
\begin{align*}
\phi_\mathfrak{g} [x,y] = [\phi_\mathfrak{g}x, y] + [ x, \phi_\mathfrak{g} y ], ~~~ \text{ for all } x, y \in \mathfrak{g}.
\end{align*}
In \cite{tang-f-sheng} the authors call the pair $(\mathfrak{g}, \phi_\mathfrak{g})$ of a Lie algebra and a derivation by a LieDer pair. Here we introduce higher derivations on a Lie algebra similar to the case of associative algebras.

\begin{defn}
A {\em higher derivation (of rank $N$)} on a Lie algebra $(\mathfrak{g}, [~,~])$ consists of a finite sequence $\{ \phi_1, \ldots, \phi_N \}$ of linear maps on $\mathfrak{g}$ satisfying
\begin{align*}
\phi_k [x, y] = \sum_{i+j = k} [\phi_i (x), \phi_j (y)], ~~~ \text{ for } k =1, \ldots, N \text{ and } x, y \in \mathfrak{g}.
\end{align*}
\end{defn}
We call the tuple $(\mathfrak{g}, \phi_1, \ldots, \phi_N)$ of a Lie algebra and a higher derivation on it by a LieHDer pair of rank $N$. Morphism between LieHDer pairs of same rank can be defined similarly. We denote the category of LieHDer pairs by {\bf LieHDer}.

Let $(A, d_1, \ldots, d_N)$ be an AssHDer pair. Consider the Lie algebra structure on $A$ given by the commutator bracket $[a, b]_c = ab - ba$, for $a,b \in A$. We denote this Lie algebra by $A_c$. It can be easily verify that the sequence $\{ d_1, \ldots, d_N \}$ is a higher derivation of rank $N$ on the Lie algebra $A_c$ as
\begin{align*}
d_k [a, b]_c = d_k (ab -ba) =~& \sum_{i+j = k} ( d_i (a) d_j (b) - d_j (b) d_i (a) )\\
=~& \sum_{i+j = k} [d_i (a), d_j (b) ], ~~~ \text{ for } k = 1, \ldots, N.
\end{align*}
Therefore, it defines a functor $(~)_c : {\bf AssHDer} \rightarrow {\bf LieHDer}$ from the category of AssHDer pairs to the category of LieHDer pairs. In the following, we construct a functor $U : {\bf LieHDer} \rightarrow {\bf AssHDer}$ left adjoint to the functor $(~)_c$.

Let $(\mathfrak{g}, \phi_1, \ldots, \phi_N)$ be a LieHDer pair of rank $N$. Consider the universal enveloping associative algebra $U (\mathfrak{g})$ of $\mathfrak{g}$. By definition, $U(\mathfrak{g})$ is the quotient of the tensor algebra $T(\mathfrak{g})$ by the two sided ideal generated by the elements of the form $x \otimes y - y \otimes x - [x, y]$, for $x, y \in \mathfrak{g}$. Note that the linear maps $\overline{\phi}_1, \ldots, \overline{\phi}_N : T(\mathfrak{g}) \rightarrow T(\mathfrak{g})$ (see Proposition \ref{free-asshd}) induces maps $\overline{\overline{\phi}}_1, \ldots, \overline {\overline {\phi}}_N $ on $U(\mathfrak{g})$ as
\begin{align*}
\overline{\phi}_k ( x \otimes y - y \otimes x - [x, y] ) = \sum_{i+j = k} ( \phi_i (x) \otimes \phi_j (y) - \phi_j (y) \otimes \phi_i (x) - [ \phi_i (x), \phi_j (y)] ).
\end{align*}
Moreover, the maps $\{ \overline{\overline{\phi}}_1, \ldots, \overline {\overline {\phi}}_N  \}$ constitute a higher derivation on $U(\mathfrak{g}).$ In other words, $\{ U(\mathfrak{g}), \overline{\overline{\phi}}_1, \ldots, \overline {\overline {\phi}}_N \}$ is an AssHDer pair of rank $N$. This construction is also functorial, hence, defines a functor $U : {\bf LieHDer} \rightarrow {\bf AssHDer}$ from the category of LieHDer pairs to the category of AssHDer pairs.

\begin{prop}
The functor $U : {\bf LieHDer} \rightarrow {\bf AssHDer}$ is left adjoint to the functor $(~)_c : {\bf AssHDer} \rightarrow {\bf LieHDer}$.
\end{prop}

\begin{proof}
We have to prove that there is an isomorphism
\begin{align*}
\mathrm{Hom}_{\bf AssHDer} ( U(\mathfrak{g}), A) \cong \mathrm{Hom}_{\bf LieHDer} ( \mathfrak{g}, A_c),
\end{align*}
for any AssHDer pair $(A, d_1, \ldots, d_N)$ and any LieHDer pair $(\mathfrak{g}, \phi_1, \ldots, \phi_N)$. Take any AssHDer pair morphism $f : U(\mathfrak{g}) \rightarrow A$. Its restriction to $\mathfrak{g}$ is a Lie algebra morphism $\mathfrak{g} \rightarrow A_c$ and commute with higher derivations. Hence it is a morphism of LieHDer pairs.

Conversely, for any LieHDer pair morphism $h : \mathfrak{g} \rightarrow A_c$, one can extend it to an AssHDer pair morphism $\widetilde{h} : T(\mathfrak{g}) \rightarrow A$. It induces a morphism $\widetilde{\widetilde{h}} : U (\mathfrak{g}) \rightarrow A$ of AssHDer pairs as $h$ is a LieHDer pair morphism. Finally, it is easy to see that the above two correspondences are inverses to each other. Hence the proof.
\end{proof}

\section{Representations and Cohomology}
In this section, we define representations and cohomology for an AssHDer pair. This generalizes the cohomology of AssDer pair as introduced in \cite{das-mandal}.

\begin{defn}
Let $(A, d_1, \ldots, d_N )$ be an AssHDer pair. A left module over it consists of a tuple $(M, d_1^M, \ldots, d_N^M )$ in which $M$ is a left $A$-module and $d_1^M, \ldots, d_N^M $ are linear maps on $M$ satisfying
\begin{align*}
d_k^M ( am ) = \sum_{i+j = k} d_i (a) d_j^M (m), ~~~~ \text{ for } k = 1, \ldots, N
\end{align*}
with convention that $d_0 = \mathrm{id}_A$ and $d_0^M = \mathrm{id}_M$. 
\end{defn}
Similarly, a right module over $(A, d_1, \ldots, d_N )$ consists of a tuple $(M, d_1^M, \ldots, d_N^M )$ in which $M$ is a right $A$-module and  $d_1^M, \ldots, d_N^M $ are linear maps on $M$ satisfying
\begin{align*}
d_k^M (ma) = \sum_{i+j = k } d_i^M (m) d_j (a), ~~~~ \text{ for } k = 1, \ldots, N.
\end{align*}
A bimodule (representation) over $(A, d_1, \ldots, d_N )$ is a tuple $(M, d_1^M, \ldots, d_N^M )$ which is both left module and right module and $M$ is additionally an $A$-bimodule, i.e. $(am)b = a (mb),$ for $a, b \in A$ and $m \in M$.

It has been proved in \cite{das-mandal} that given a representation of an AssDer pair gives rise to a semi-direct product. We extend this result in the context of higher derivations.

\begin{prop}\label{semi-prop}
Let $(A, d_1, \ldots, d_N)$ be an AssHDer pair of rank $N$ and $(M, d_1^M, \ldots, d_N^M )$ be a representation of it. Then $(A \oplus M, d_1 \oplus d_1^M, \ldots, d_N \oplus d_N^M)$ is an AssHDer pair where the associative algebra structure on $A \oplus M$ is given by the semi-direct product
\begin{align*}
(a , m ) (b, n) = (ab, an + m b).
\end{align*}
\end{prop}

\begin{proof}
It is known that $A \oplus M$ equipped with the above product is an associative algebra \cite{loday-book}. Moreover, for any $k \in \{ 1, \ldots, N \}$, we have
\begin{align*}
(d_k \oplus d_k^M) ( (a,m)(b,n)) =~& ( d_k (ab), d_k^M (an) + d_k^M (mb)) \\
=~& \sum_{i+j = k} ( d_i(a) d_j (b), d_i (a) d_j^M (n) + d_i^M (m) d_j (b) ) \\
=~&  \sum_{i+j = k} (d_i (a), d_i^M (m)) (d_j (b), d_j^M (n))\\
=~&  \sum_{i+j = k} (d_i \oplus d_i^M)(a, m) (d_j \oplus d_j^M)(b,n).
\end{align*}
Hence the proof.
\end{proof}


One may also define representations of LieHDer pairs. Let $(\mathfrak{g}, \phi_1, \ldots, \phi_N)$ be a LieHDer pair of rank $N$. A representation over it consists of a tuple $(M, \phi_1^M, \ldots, \phi_N^M)$ in which $M$ is a $\mathfrak{g}$-module and $\phi_1^M, \ldots, \phi_N^M$ are linear maps on $M$ satisfying
\begin{align*}
\phi_k^M [x, m] = \sum_{i+j = k} [ \phi_i (x), \phi_j^M(m)], ~~~ \text{ for } x \in \mathfrak{g}, m \in M \text{ and }  k =1, \ldots, N.
\end{align*}

\begin{prop}
Let $(\mathfrak{g}, \phi_1, \ldots, \phi_N)$ be a LieHDer pair. A representation over it is equivalent to a left module over the AssHDer pair $(U(\mathfrak{g}), \overline{\overline{\phi}}_1, \ldots, \overline {\overline {\phi}}_N).$
\end{prop}

\begin{proof}
It is known that a $\mathfrak{g}$-module is equivalent to a left $U(\mathfrak{g})$-module. More precisely, let $M$ be a left $U(\mathfrak{g})$-module, then the $\mathfrak{g}$-module structure on $M$ is given by $[x, m] := xm$, for $x \in \mathfrak{g}$ and $m \in M$.

Let $(M, \phi_1^M, \ldots, \phi_N^M)$ be a left module over the AssHDer pair $(U(\mathfrak{g}), \overline{\overline{\phi}}_1, \ldots, \overline {\overline {\phi}}_N)$. Then the conditions
\begin{align*}
\phi_k^M (xm ) = \sum_{i+j = k } \overline{\overline{\phi}}_i (x) \phi_j^M (m), ~~~~ \text{ for } k =1, \ldots, N
\end{align*}
are equivalent to $\phi_k^M [x, m] = \sum_{i+j = k} [\phi_i (x), \phi_j^M (m)],$ for $k =1, \ldots, N$. Hence the result follows.
\end{proof}

In the next, we introduce a cohomology for an AssHDer pair with coefficients in a representation. Before that, we introduce some notations.

Let $(A, d_1, \ldots, d_N)$ be an AssHDer pair of rank $N$ and $(M, d_1^M, \ldots, d_N^M)$ be a representation of it. For any $n \geq 1$, we define some new maps $\delta_k : \mathrm{Hom}(A^{\otimes n}, M ) \rightarrow \mathrm{Hom}(A^{\otimes n}, M )$, for $k = 1, \ldots, N$, by
\begin{align*}
\delta_k f = \sum_{i_1 + \cdots + i_n = k} f \circ ( d_{i_1} \otimes \cdots \otimes d_{i_n} ) - d_k^M \circ f.
\end{align*}
Moreover, for any $N$-tuple $(f_1, \ldots, f_N)$ of maps in $\mathrm{Hom}(A^{\otimes n}, M)$, we define $\delta_{\mathrm{Hoch}}' (f_k ) \in \mathrm{Hom}(A^{\otimes n+1}, M)$, for $k =1, \ldots, N$, by
\begin{align*}
\delta_{\mathrm{Hoch}}' (f_k ) (a_1, \ldots, a_{n+1}) =~& \sum_{i+j = k} d_i (a_1) f_j (a_2, \ldots, a_{n+1}) + \sum_{i=1}^n (-1)^i~ f_k (a_1, \ldots, a_i a_{i+1}, \ldots, a_{n+1} ) \\
~&+ (-1)^{n+1} \sum_{i+j = k} f_i (a_1, \ldots, a_n ) d_j (a_{n+1}), ~~ \quad \text{ for } k =1, \ldots, N.
\end{align*}
The operator $\delta_{\mathrm{Hoch}}'$ is a modification of the classical Hochschild coboundary operator $\delta_{\mathrm{Hoch}}$.

\begin{remark}
When $(M, d_1^M, \ldots, d_N^M) = (A, d_1, \ldots, d_N)$ there is a bracket generalizing the Gerstenhaber bracket \cite{gers} as follows. For any $P \in \mathrm{Hom}( A^{\otimes m }, A)$ and a tuple $(f_1, \ldots, f_N)$ of maps in $\mathrm{Hom} (A^{\otimes n}, A)$, we define $[P, f_k]_N \in \mathrm{Hom} ( A^{\otimes m+n-1}, A),$ for $k=1, \ldots, N$, by
\begin{align*}
&[P, f_k]_N (a_1, \ldots, a_{m+n-1}) \\&:= \sum_{j=1}^m \sum_{i_1 + \cdots + i_m = k} (-1)^{(j-1)(n-1)} ~P \big( d_{i_1} (a_1), \ldots, f_{i_j} (a_j, \ldots, a_{j+n-1}), \ldots, d_{i_m} (a_{m+n-1}) \big) \\
& \quad - (-1)^{(m-1)(n-1)} \sum_{j=1}^n (-1)^{(j-1)(m-1)} f_k \big( a_1, \ldots, a_{j-1}, P (a_j, \ldots, a_{j+m-1}), \ldots, a_{m+n-1} \big).
\end{align*}
Define $[f_k, P]_N := - (-1)^{(m-1)(n-1)} [P, f_k]$. This bracket satisfy similar properties as of Gerstenhaber bracket. With this set-up, we have
\begin{align*}
\delta_k f = - [ d_k, f]_N  \quad \text{ and } \quad \delta_{\mathrm{Hoch}}' (f_k) = (-1)^{n-1} [\mu, f_k ]_N.
\end{align*}
\end{remark}

The following two results are useful to define the coboundary operator of the cohomology of AssHDer pair.
\begin{prop}
The operator $\delta_{\mathrm{Hoch}}'$ satisfies $(\delta_{\mathrm{Hoch}}')^2 = 0$.
\end{prop}

\begin{lemma}
We have $\delta_{\mathrm{Hoch}} ' \circ \delta_k = \delta_k \circ \delta_{\mathrm{Hoch}}$, for $k=1, \ldots, N$.
\end{lemma}

We are now in a position to define the cohomology of the AssHDer pair. Define the space $C^0_{\mathrm{AssHDer}} (A, M)$ of $0$-cochains to be $0$ and the space $C^1_{\mathrm{AssHDer}} (A, M)$ of $1$-cochains to be $\mathrm{Hom}(A, M)$. The space of $n$-cochains $C^n_{\mathrm{AssHDer}}(A, M)$, for $n \geq 2$, to be defined by
\begin{align*}
C^n_{\mathrm{AssHDer}}(A, M) := \mathrm{Hom}( A^{\otimes n}, M) \times \underbrace{\mathrm{Hom}( A^{\otimes n-1}, M) \times \cdots \times \mathrm{Hom}( A^{\otimes n-1}, M)}_{N ~\text{times}}.
\end{align*}
We define a map $\partial : C^n_{\mathrm{AssHDer}}(A, M) \rightarrow C^{n+1}_{\mathrm{AssHDer}}(A, M)$ by
\begin{align*}
\partial f = ( \delta_{\mathrm{Hoch}} f ; - \delta_1 f, -\delta_2 f, \ldots, - \delta_N f),~~~ \text{ for } f \in C^1_{\mathrm{AssHDer}}(A, M),
\end{align*}
\begin{align*}
\partial ( f; f_1, \ldots, f_N ) = ( \delta_{\mathrm{Hoch}} f ;~ \delta_{\mathrm{Hoch}}' f_1 + (-1)^n \delta_1 f,~ \delta_{\mathrm{Hoch}}' f_2 + (-1)^n \delta_2 f, \ldots,  \delta_{\mathrm{Hoch}}' f_N + (-1)^n \delta_N f),
\end{align*}
for $( f; f_1, \ldots, f_N ) \in C^n_{\mathrm{AssHDer}}(A, M), n \geq 2$.

\begin{prop}
The map $\partial$ satisfies $\partial^2 = 0$.
\end{prop}

\begin{proof}
For any $f \in C^1_{\mathrm{AssHDer}} (A, M)$, we have
\begin{align*}
\partial^2 f =~& \partial ( \delta_{\mathrm{Hoch}} f;~  - \delta_1 f, \ldots, - \delta_N f) \\
=~& ( (\delta_{\mathrm{Hoch}})^2 f;~ - \delta_{\mathrm{Hoch}}' \delta_1 f + \delta_1 \delta_{\mathrm{Hoch}} f, \ldots, - \delta_{\mathrm{Hoch}}' \delta_N f + \delta_N \delta_{\mathrm{Hoch}} f ) = 0.
\end{align*}
Similarly, for any $(f; f_1, \ldots, f_N ) \in C^n_{\mathrm{AssHDer}} (A, M)$ with $n \geq 2$, we have
\begin{align*}
~&\partial^2 (f; f_1, \ldots, f_N )\\ =~& \partial ( \delta_{\mathrm{Hoch}} f;~ \delta_{\mathrm{Hoch}}' f_1 + (-1)^n \delta_1 f, \ldots, \delta_{\mathrm{Hoch}}' f_N + (-1)^n \delta_N f) \\
=~& \big( (\delta_{\mathrm{Hoch}})^2 f;~ (\delta_{\mathrm{Hoch}}')^2 f_1 + (-1)^n \delta_{\mathrm{Hoch}}' \delta_1 f +(-1)^{n+1} \delta_1 \delta_{\mathrm{Hoch}} f, \ldots,\\& \qquad \qquad \qquad (\delta_{\mathrm{Hoch}}')^2 f_N + (-1)^n \delta_{\mathrm{Hoch}}' \delta_N f +(-1)^{n+1} \delta_N \delta_{\mathrm{Hoch}} f \big) \\
=~& 0.
\end{align*}
Hence the proof.
\end{proof}

Therefore, $\{ C^\ast_{\mathrm{AssHDer}}(A, M), \partial \}$ forms a cochain complex.
We denote the corresponding cohomology groups by $H^\ast_{\mathrm{AssHDer}} (A, M)$.

\section{Central extensions}
In this section, we study extensions of an AssHDer pair by a trivial AssHDer pair. The main result of this section states that isomorphism classes of such extensions are classified by the second cohomology of the AssHDer pair with coefficients in the trivial representation.

Let $(A, d_1, \ldots, d_N)$ be an AssHDer pair and $(M, d_1^M, \ldots, d_N^M)$ be a trivial AssHDer pair (i.e. the associative product of $M$ is trivial).

\begin{defn}
A central extension of $(A, d_1, \ldots, d_N)$ by $(M, d_1^M, \ldots, d_N^M)$ is an exact sequence of AssHDer pairs
\begin{align}\label{eqn-t}
\xymatrix{
0 \ar[r]&  (M, d_1^M, \ldots, d_N^M) \ar[r]^{i} & (\widehat{A}, \widehat{d}_1, \ldots, \widehat{d}_N) \ar[r]^{p} & (A, d_1, \ldots, d_N) \ar[r] & 0
}
\end{align}
satisfying $i(m) \cdot \widehat{a} = 0 = \widehat{a} \cdot i(m)$, for all $m \in M$ and $\widehat{a} \in \widehat{ A}.$
\end{defn}

In a central extension, we can identify $M$ with the corresponding subalgebra of $\widehat{ A}$ and with this identification $d_k^M = \widehat{d}_k |_M,$ for all $k =1, \ldots, N$.

\begin{defn}
Two central extensions $(\widehat{A}, \widehat{d}_1, \ldots, \widehat{d}_N)$ and $(\widehat{A}', \widehat{d}_1', \ldots, \widehat{d}_N')$ are said to be isomorphic if there is an isomorphism $\eta : (\widehat{A}, \widehat{d}_1, \ldots, \widehat{d}_N) \rightarrow (\widehat{A}', \widehat{d}_1', \ldots, \widehat{d}_N') $ of AssHDer pairs that makes the following diagram commutative
\[
\xymatrix{
0 \ar[r] & (M, d_1^M, \ldots, d_{N}^M) \ar@{=}[d] \ar[r]^i & (\widehat{A}, \widehat{d}_1, \ldots, \widehat{d}_N) \ar[r]^p \ar[d]_\eta & (A, d_1, \ldots, d_N) \ar@{=}[d] \ar[r] & 0 \\
0 \ar[r] & (M, d_1^M, \ldots, d_{N}^M) \ar[r]_{i'} & (\widehat{A}', \widehat{d}_1', \ldots, \widehat{d}_N') \ar[r]_{p'} & (A, d_1, \ldots, d_N) \ar[r] & 0. \\
}
\]
\end{defn}

Let (\ref{eqn-t}) be a central extension. A section of the map $p$ is given by a linear map $s : A \rightarrow \widehat{A}$ such that $p \circ s = \mathrm{id}_A$.

For any section $s$, we define a map $\psi : A^{\otimes 2} \rightarrow M$ and maps $\chi_k : A \rightarrow M $ (for $k = 1, \ldots, N$) by
\begin{align*}
\psi (a, b) = s(a) \cdot s(b) - s(ab) ~~~ \text{ and } ~~~ \chi_k (a) = \widehat{d}_k ( s(a)) - s ( d_k (a)).
\end{align*}
Note that the vector space $\widehat{A}$ is isomorphic to the direct sum $A \oplus M$ via the section $s$. Therefore, we may transfer the structures (product and linear maps) of $\widehat{A}$ to $A \oplus M$. The product and linear maps on $A \oplus M$ are given by
\begin{align*}
(a \oplus m) \cdot (b \oplus n) =~& ab \oplus \psi (a, b),\\
d_k^{A \oplus M} (a \oplus m) =~& d_k (a) \oplus d_k^M (m) + \chi_k (a), ~~~ \text{ for } k = 1, \ldots, N.
\end{align*}

\begin{prop}\label{prop-3}
The vector space $A \oplus M$ equipped with the above product and linear maps $d_k^{A \oplus M}$ (for $k =1 , \ldots, N$) forms an AssHDer pair if and only if $(\psi; \chi_1, \ldots, \chi_N)$ is a $2$-cocycle in the cohomology of the AssHDer pair $A$ with coefficients in the trivial representation $M$. Moreover, the cohomology class of $(\psi; \chi_1, \ldots, \chi_N)$ does not depend on the choice of the section $s$.
\end{prop}

\begin{proof}
The tuple $(A \oplus M, d_1^{A \oplus M}, \ldots, d_N^{A \oplus M})$ is an AssHDer pair if and only if
\begin{align}
((a \oplus m) \cdot (b \oplus n)) \cdot (c \oplus p) = ( a \oplus m ) \cdot ( (b \oplus n) \cdot (n \oplus p)), \label{prop-3-1}\\
d_k^{A \oplus M} ((a \oplus m) \cdot (b \oplus n)) = \sum_{i+ j = k} d_i^{A \oplus M} ( a \oplus m) \cdot d_j^{A \oplus M} (b \oplus n), \label{prop-3-2}
\end{align}
for all $a \oplus m,~ b \oplus n,~ c \oplus p \in A \oplus M$ and $k = 1, \ldots, N$. The condition (\ref{prop-3-1}) is equivalent to $\psi (ab, c) = \psi (a, bc)$, or, equivalently, $\delta_{\mathrm{Hoch}} \psi = 0$. The condition (\ref{prop-3-2}) is equivalent to
\begin{align*}
d_k^M ( \psi (a, b)) + \chi_k (ab) = \sum_{i+j = k} \psi ( d_i (a), d_j (b)), ~~~ \text{ for } k = 1, \ldots, N.
\end{align*}
This is same as $\delta_{\mathrm{Hoch}}' \chi_k + \delta_k \psi = 0$, for all $k =1,\ldots, N$. Therefore, (\ref{prop-3-1}) and (\ref{prop-3-2}) holds if and only if $\partial ( \psi; \chi_1, \ldots, \chi_N) = 0$.

Let $s_1, s_2$ be two sections of $p$. Define a map $\phi :A \rightarrow M$ by $\phi (a) := s_1 (a) - s_2 (a).$ Then we have
\begin{align*}
\psi^1 (a, b) = ~& s_1 (a) \cdot s_1(b) - s_1 (ab) \\
=~& ( s_2 (a) + \phi (a) ) \cdot ( s_2 (b) + \phi (b)) - s_2 (ab) - \phi (ab) \\
=~& \psi^2 (a, b) - \phi (ab) \qquad (\text{as } \phi(a) , \phi (b) \in M)
\end{align*}
and
\begin{align*}
\chi_k^1 (a) =~& \widehat{d}_k (s_1 (a)) - s_1 ( d_k (a)) \\
=~& \widehat{d}_k ( s_2 (a) + \phi (a)) - s_2 ( d_k (a)) - \phi ( \phi_A (a)) \\
=~& \chi^2_k (a) + d^M_k ( \phi (a)) - \phi (d_k (a)), \quad \text{ for } k = 1, \ldots, N.
\end{align*}
This shows that $(\psi^1, \chi_1^1, \ldots, \chi_N^1) - ( \psi^2, \chi^2_1, \ldots, \chi^2_N) = \partial \phi$. Hence they correspond to the same cohomology class.
\end{proof}

\begin{thm}\label{cent-thm}
Let $(A, d_1, \ldots, d_n)$ be an AssHDer pair and $(M, d_1^M, \ldots, d_N^M)$ be a trivial AssHDer pair. Then the isomorphism classes of central extensions of $A$ by $M$ are classified by the second cohomology group $H^2_{\mathrm{AssHDer}} (A, M).$
\end{thm}

\begin{proof}
Let $(\widehat{A}, \widehat{d}_1, \ldots, \widehat{d}_N)$ and $(\widehat{A}', \widehat{d}_1', \ldots, \widehat{d}_N')$ be two isomorphic central extensions and the isomorphism is given by $\eta : \widehat{A} \rightarrow \widehat{A}'$. Let $s : A \rightarrow \widehat{A}$ be a section of $p$. Then
\begin{align*}
p ' \circ ( \eta \circ s) = (p' \circ \eta ) \circ s = p \circ s = \mathrm{id}_A.
\end{align*}
This shows that $s' := \eta \circ s$ is a section of $p'$. Since $\eta$ is a morphism of AssHDer pairs, we have $\eta|_M = \mathrm{id}_M$. Thus,
\begin{align*}
\psi' (a, b) = s'(a) \cdot s'(b) - s'(ab)  =~& (\eta \circ s)(a) \cdot (\eta \circ s) (b) - (\eta \circ s) (ab) \\
=~& s(a) \cdot s(b) - s( ab) = \psi (a, b)
\end{align*}
and
\begin{align*}
\chi_k' (a) = d_k^{\widehat{A}'} ( s'(a)) - s' ( d_k (a)) 
=~& d_k^{\widehat{A}'} ( \eta \circ s (a)) - \eta \circ s ( d_k (a)) \\
=~& d_k^{\widehat{A}} ( s(a)) - s ( d_k (a)) = \chi_k (a).
\end{align*}
Therefore, isomorphic central extensions give rise to same $2$-cocycle, hence, correspond to same element in $H^2_{\mathrm{AssHDer}}(A, M).$

Conversely, let $(\psi; \chi_1, \ldots, \chi_N)$ and $(\psi'; \chi_1', \ldots, \chi_N')$ be two cohomologous $2$-cocycles. Therefore, there exists a map $\phi : A \rightarrow M$ such that
\begin{align*}
(\psi; \chi_1, \ldots, \chi_N) - (\psi'; \chi_1', \ldots, \chi_N') = \partial \phi.
\end{align*}
The AssHDer pair structures on $A \oplus M$ (given in Proposition \ref{prop-3}) corresponding to the above $2$-cocycles are isomorphic via the map $\eta : A \oplus M \rightarrow A \oplus M$ given by $\eta ( a \oplus m) = a \oplus m + \phi(a)$. Hence the proof.
\end{proof}

\section{Abelian extensions}
It is known that equivalence classes of abelian extensions of an associative algebra by a bimodule are classified by the second Hochschild cohomology group \cite{loday-book}. In this section, we generalize this result to AssHDer pairs.

Let $(A, d_1, \ldots, d_N )$ be an AssHDer pair of rank $N$. Suppose $(M, d_1^M, \ldots, d_N^M)$ is a vector space equipped with linear maps. It can be considered as an AssHDer pair in which $M$ is given by the trivial multiplication.

\begin{defn}
An abelian extension of $(A, d_1, \ldots, d_N)$ by $(M, d_1^M, \ldots, d_N^M)$ is an exact sequence of AssHDer pairs
\begin{align}\label{abel-ext-eqn}
\xymatrix{
0 \ar[r] & (M, d_1^M, \ldots, d_N^M) \ar[r]^{i} & (E, d_1^E, \ldots, d_N^M) \ar[r]^{p} & (A, d_1, \ldots, d_N) \ar[r] & 0
}
\end{align}
together with a $\mathbb{K}$-splitting $s : A \rightarrow E$.
\end{defn}

Such an extension induces a representation of $(A, d_1, \ldots, d_N)$ on $(M, d_1^M, \ldots, d_N^M)$ given by $am = \mu_E (s(a), i(m))$ and $ma = \mu_E ( i(m), s(a))$, where $\mu_E$ denotes the associative multiplication on $E$. This action is independent of the section $s$.

\begin{defn}
Two abelian extensions $(E, d_1^E, \ldots, d_N^E)$ and $(E', d_1^{E'}, \ldots, d_N^{E'})$ are said to be  equivalent if there is a morphism $\Psi : E \rightarrow E'$ of AssHDer pairs which makes the following diagram commutative
\[
\xymatrix{
0 \ar[r] & (M, d_1^M, \ldots, d_{N}^M) \ar[r] \ar@{=}[d] & (E, d_1^E, \ldots, d_N^E) \ar[r]^p \ar[d]_\Psi & (A, d_1, \ldots, d_N) \ar@{=}[d] \ar[r] & 0 \\
0 \ar[r] & (M, d_1^M, \ldots, d_{N}^M) \ar[r] & (E', d_1^{E'}, \ldots, d_N^{E'}) \ar[r]_{p'} & (A, d_1, \ldots, d_N) \ar[r] & 0.
}
\]
\end{defn}

Let $(A, d_1, \ldots, d_N)$ be an AssHDer pair and $(M, d_1^M, \ldots, d_N^M)$ be a representation of it. We denote by $Ext (A, M)$ the set of equivalence classes of abelian extensions in which the induced representation of $(A, d_1, \ldots, d_N)$ on $(M, d_1^M, \ldots, d_N^M)$ is the given one.

\begin{thm}\label{abel-cla}
There is a one-to-one correspondence $H^2_{\mathrm{AssHDer}}(A, M ) \cong Ext(A, M).$
\end{thm}

\begin{proof}
Let $(f; f_1, \ldots, f_N) \in C^2_{\mathrm{AssHDer}} (A, M)$ be a $2$-cocycle. In other words,  we  have $\delta_{\mathrm{Hoch}} f = 0 $ and $\delta_{\mathrm{Hoch}}' f_k + \delta_k f =0$, for all $k =1, \ldots, N$. Consider the vector space $E = A \oplus M$ with the following structure maps
\begin{align*}
\mu_E ((a, m), (b,n)) =~& ( ab, an + mb + f(a,b))   ~~~~\text{ and }\\
d_k^E ((a,m)) =~& ( d_k (a), d_k^M (m) + f_k (a)),
\end{align*}
for $k =1, \ldots, N$. (Observe that when the $2$-cocycle  $(f; f_1, \ldots, f_N)$ is zero, this is the semi-direct product. See Proposition \ref{semi-prop}.) The multiplication $\mu_E$ is associative as $\delta_{\mathrm{Hoch}} f =0$. Moreover $\{ d_1^E, \ldots, d_N^E \}$ is a higher derivation on $E$. Hence $0 \rightarrow M \rightarrow E \rightarrow A \rightarrow 0$ is an abelian extension of AssHDer pair  with the obvious splitting.

Let $(f; f_1, \ldots, f_N)$ and $(f'; f_1', \ldots, f_N')$ be two cohomologous $2$-cocycles. Then there exists $h \in C^1_{\mathrm{AssHDer}} (A, M) = \mathrm{Hom} (A, M)$ such that
\begin{align*}
(f; f_1, \ldots, f_N) - (f'; f_1', \ldots, f_N') = \partial h.
\end{align*}
The corresponding AssHDer pair structures are related by the equivalence $\Psi : E \rightarrow E', ~ (a, m) \mapsto (a, m + h (a))$. Hence the map $H^2_{\mathrm{AssHDer}}(A, M ) \rightarrow Ext(A, M)$ is well defined.

\medskip

Conversely, given any abelian extension (\ref{abel-ext-eqn}) with splitting $s$, the vector space $E$ is isomorphic to $A \oplus M$ and with this identification $s(a) = (a, 0)$. The maps $i$ and $p$ are the obvious ones with respect to the above splitting. The map $p$ is an algebra map implies that $p \circ \mu_E ((a,0), (b, 0)) = ab$. This implies that $\mu_E ((a, 0), (b, 0)) = (ab, f(a, b))$, for some $f \in \mathrm{Hom}(A^{\otimes 2}, M)$. The associativity of $\mu_E$ implies that $\delta_{\mathrm{Hoch}} f = 0$. Moreover, since $p \circ d_k^E = d_k \circ p$, we have $d_k^E (a, 0) = (d_k (a), f_k (a)) $, for some $f_k \in \mathrm{Hom} (A, M),$ for $k =1, \ldots, N$. The sequence $\{ d_1^E, \ldots, d_N^E \}$ is a higher derivation implies that
\begin{align*}
\delta_{\mathrm{Hoch}}' f_k + \delta_k f = 0, ~~~ \text{ for } k =1, \ldots, N.
\end{align*}
Hence $(f; f_1, \ldots, f_N) \in C^2_{\mathrm{AssHDer}} (A, M)$ is a $2$-cocycle.

Similarly, one can show that any two equivalent extensions $(E, d_1^E, \ldots, d_N^E)$ and $(E', d_1^{E'}, \ldots, d_N^{E'})$ are related by a map $E = A \oplus M \xrightarrow{\Psi } A \oplus M = E'$ defined by $\Psi (a, m) = (a, m + h(a)),$ for some $h \in \mathrm{Hom}(A, M)$. Since $\Psi$ is a morphism of AssHDer pairs, we have $(f; f_1, \ldots, f_N) - (f'; f_1', \ldots, f_N') = \partial h$. Here $(f'; f_1', \ldots, f_N') \in C^2_{\mathrm{AssHDer}} (A, M)$ is the $2$-cocycle induced from the extension $E'$. Therefore, the map $Ext(A, M) \rightarrow H^2_{\mathrm{AssHDer}} (A, M)$ is also well defined. Finally, the above two maps are inverses to each other. Hence the proof.
\end{proof}

\section{Formal deformations}

The classical deformation theory of Gerstenhaber for associative algebras has been extended to AssDer pairs in \cite{das-mandal}. In this section, we further extend this to AssHDer pairs.

Let $(A, d_1, \ldots, d_N)$ be an AssHDer pair of rank $N$. We denote the associative multiplication on $A$ by $\mu$. Consider the space $A[[t]]$ of formal power series in $t$ with coefficients from $A$. Then $A[[t]]$ is a $\mathbb{K}[[t]]$-module and when $A$ is finite dimensional, we have $A[[t]] \cong A \otimes_\mathbb{K} \mathbb{K}[[t]]$.

A formal one-parameter deformation of the AssHDer pair $(A, d_1, \ldots, d_N)$ consist of formal power series
\begin{align*}
\mu_t =~& \mu_0 +  t \mu_1  + t^2 \mu_2 + \cdots ~~~~ \in \mathrm{Hom}(A^{\otimes 2}, A) [[t]] ~~~\text{ with } \mu_0 = \mu,\\
d_{1, t} =~& d_{1, 0} + t d_{1,1}+ t^2 d_{1,2} + \cdots ~~~~ \in \mathrm{Hom}(A, A) [[t]] ~~~ \text{ with } d_{1,0} = d_1,\\
\vdots \\
d_{N, t} =~& d_{N, 0} + t d_{N, 1} + t^2 d_{N, 2} + \cdots  ~~~~  \in \mathrm{Hom}(A, A) [[t]] ~~~ \text{ with } d_{N,0} = d_N
\end{align*}
such that $(A[[t]], \mu_t )$ is an associative algebra over $\mathbb{K}[[t]]$ and $\{ d_{1,t}, \ldots, d_{N, t} \}$ constitute a higher derivation of rank $N$ on it.

\begin{remark}
View the higher derivation $\{ d_1, \ldots, d_N \}$ on $A$ as an algebra morphism $A \rightarrow A[[t]]/ (t^{N+1})$ (see Remark \ref{higher-der-as-mor}). First observe that a deformation of the algebra $A$ induces a deformation of the algebra $A[[t]]/ (t^{N+1})$. With this view-point, a deformation of the AssHDer pair $(A, d_1, \ldots, d_N)$ is a deformation of the corresponding algebra morphism $A \rightarrow A[[t]]/ (t^{N+1})$ in the sense of Gerstenhaber and Schack \cite{gers-sch}.
\end{remark}

Therefore, in a formal one-parameter deformation of AssHDer pair, the following relations hold:
\begin{align}
\mu_t ( \mu_t (a, b), c) =~ \mu_t (a, \mu_t (b, c)) \quad \text{ and } \quad 
d_{k, t} (\mu_t (a, b)) =~ \sum_{i+j = k} \mu_t (d_{i, t} (a) , d_{j, t}(b)),
\end{align}
for $k = 1, \ldots, N$. The above identities are equivalent to 
\begin{align*}
\sum_{i+j = n} \mu_i ( \mu_j(a, b), c) =~& \sum_{i+j = n} \mu_i ( a, \mu_j (b,c)), ~~~ \text{ for } n \geq 0,\\
\sum_{i+j = n} d_{k, i} (\mu_j (a, b)) =~& \sum_{i+j = k} \sum_{p +q + r = n} \mu_p ( d_{i, q} (a), d_{j, r} (b)), ~~~ \text{ for } n \geq 0 \text{ and } k = 1, \ldots, N.
\end{align*}
All the identities hold for $n=0$ as $(A, d_1, \ldots, d_N)$ is an AssHDer pair. For $n =1$, we have
\begin{align}
\mu_1 (a, b) c + \mu_1 (ab, c) =~&  a \mu_1 (b, c) + \mu_1 (a, bc), \label{1-co}\\
d_k ( \mu_1 (a, b)) + d_{k,1} (ab) =~& \sum_{i+j = k} \big[ \mu_1 ( d_i (a), d_j (b) ) + d_{i,1}(a) d_j (b) + d_i(a) d_{j, 1}(b) \big], \label{1-cod}
\end{align}
for $k =1, \ldots, N$. The condition (\ref{1-co}) is equivalent to $\delta_{\mathrm{Hoch}} (\mu_1) =0$ whereas the conditions (\ref{1-cod}) are equivalent to $\delta_{\mathrm{Hoch}}' (d_{k,1}) + \delta_k (\mu_1)$ = 0, for $k =1, \ldots, N$. Therefore, we have 
\begin{align*}
\partial ( \mu_1; d_{1,1}, \ldots, d_{N, 1}) =0.
\end{align*}
Hence, we have the following.

\begin{prop}\label{inf-co}
Let $(\mu_t; d_{1, t}, \ldots, d_{N, t})$ be a formal one-parameter deformation of an AssHDer pair $(A, d_1, \ldots, d_N)$. Then the linear term $(\mu_1; d_{1,1}, \ldots, d_{N, 1})$ is a $2$-cocycle in the cohomology of the AssHDer pair $A$ with coefficients in itself.
\end{prop}

The $2$-cocycle $(\mu_1; d_{1,1}, \ldots, d_{N, 1})$ is called the infinitesimal of the deformation. In particular, if $(\mu_1; d_{1,1}, \ldots, d_{N, 1}) = \cdots = (\mu_{n-1}; d_{1,n-1}, \ldots, d_{N, n-1}) = 0$ and $(\mu_n; d_{1,n}, \ldots, d_{N, n})$ is non-zero, then $(\mu_n; d_{1,n}, \ldots, d_{N, n})$ is a $2$-cocycle.

Next we define a notion of equivalence between formal deformations of AssHDer pairs.

\begin{defn}
Let $(\mu_t; d_{1,t}, \ldots, d_{N, t})$ and $(\mu_t'; d_{1,t}', \ldots, d_{N, t}')$ be two deformations of an AssHDer pair $(A, d_1, \ldots, d_N)$. They are said to be equivalent if there exists a formal isomorphism $\Phi_t = \sum_{i \geq 0 } t^i \Phi_i : A[[t]] \rightarrow A[[t]]$ with $\Phi_0 = \mathrm{id}_A$ such that $\Phi_t$ is a morphism of AssHDer pairs from $(A[[t]], \mu_t, d_{1,t}, \ldots, d_{N, t} )$ to $(A[[t]], \mu_t', d_{1,t}', \ldots, d_{N, t}').$
\end{defn}

In other words, we must have
\begin{align*}
\Phi_t \circ \mu_t = \mu_t' \circ ( \Phi_t \otimes \Phi_t ) \quad \text{ and } \quad d_{k, t}' \circ \Phi_t = \Phi_t \circ d_{k, t}, ~~~ \text{ for } k =1, \ldots, N.
\end{align*}
They are equivalent to the following equations: for $n \geq 0$,
\begin{align*}
\sum_{i+j= n } \Phi_i \circ \mu_j =~& \sum_{p+q+r = n} \mu_p' \circ (\Phi_q \otimes \Phi_r ),\\
\sum_{i+j = n} d_{k, i}' \circ \Phi_j =~& \sum_{p+q = n} \Phi_p \circ d_{k, q}, ~~~ \text{ for } k =1, \ldots, N.
\end{align*}
The above identities hold for $n = 0$. For $n =1$, we have
\begin{align}
\mu_1 + \Phi_1 \circ \mu =~& \mu_1' + \mu \circ ( \Phi_1 \otimes \mathrm{id}) + \mu \circ ( \mathrm{id} \otimes \Phi_1 ),\label{diff-co}\\
d_k \circ \Phi_1 + d_{k,1}' =~& d_{k, 1} + \Phi_1 \circ d_k , ~~~~ \text{ for } k =1, \ldots, N. \label{diff-cod}
\end{align}
It follows from (\ref{diff-co}) and (\ref{diff-cod}) that
\begin{align*}
(\mu_1; d_{1,1}, \ldots, d_{N, 1}) - (\mu_1'; d_{1,1}', \ldots, d_{N, 1}') = \partial \Phi_1.
\end{align*}
\begin{prop}
The infinitesimals corresponding to equivalent deformations are cohomologous, hence, correspond to a same cohomology class.
\end{prop}

\begin{defn}
A deformation $(\mu_t; d_{1,t}, \ldots, d_{N, t})$ of an AssHDer pair is said to be trivial if it is equivalent to the undeformed deformation $(\mu_t' = \mu; d_{1, t}' = d_1, \ldots, d_{N, t}' = d_N).$
\end{defn}

\begin{thm}\label{2-thm}
If $H^2_{\mathrm{AssHDer}} (A, A) = 0$ then every formal deformation of the AssHDer pair $(A, d_1, \ldots, d_N)$ is trivial.
\end{thm}

\begin{proof}
Let $(\mu_t; d_{1, t}, \ldots, d_{N, t})$ be a deformation of the AssHDer pair $(A, d_1, \ldots, d_N)$. Then by Proposition \ref{inf-co}, the linear term $(\mu_1; d_{1,1}, \ldots, d_{N,1})$ is a $2$-cocycle. Therefore, it is a coboundary (follows from the hypothesis), say $(\mu_1; d_{1,1}, \ldots, d_{N,1}) = \partial \Phi_1$, for some $\Phi_1 \in C^1_{\mathrm{AssHDer}} (A, A) = \mathrm{Hom}(A, A).$

We set $\Phi_t = \mathrm{id}_A + t \Phi_1 : A [[t]] \rightarrow A[[t]]$ and define
\begin{align}\label{2-rigid}
\mu_t' = \Phi_t^{-1} \circ \mu_t \circ ( \Phi_t \otimes \Phi_t)  ~~~ \text{ and } ~~~ d_{k, t}' = \Phi_t^{-1} \circ d_{k, t} \circ \Phi_t, \text{ for } k =1, \ldots, N.
\end{align}
By definition, $(\mu_t'; d_{1,t}', \ldots, d_{N, t}')$ is equivalent to $(\mu_t; d_{1, t}, \ldots, d_{N, t}).$ Moreover, it follows from (\ref{2-rigid}) that
\begin{align*}
\mu_t' = \mu + t^2 \mu_2' + \cdots \quad \text{ and } \quad d_{k, t}' = d_k + t^2 d_{k, 2}' + \cdots, ~~~\text{ for } k =1, \ldots, N.
\end{align*}
In other words, the linear terms are vanish. By repeating this argument, we conclude the result.
\end{proof}

An AssHDer pair is said to be rigid if every deformation of it is equivalent to the undeformed one. The above theorem says that the vanishing of the second cohomology is a sufficient condition for the rigidity of an AssHDer pair.

\subsection{Finite order deformations and their extensions}

In this subsection, we consider the extension of a finite order deformation of an AssHDer pair to a deformation of next order. Given a finite order deformation, we associate a $3$-cocycle in the cohomology of the AssHDer pair, called the obstruction cocycle. When the corresponding cohomology class is trivial, the deformation extends to a deformation of next order.

Let $(A, d_1, \ldots, d_N)$ be an AssHDer pair of rank $N$. A deformation of order $n$ consist of sums $(\mu_t; d_{1, t}, \ldots, d_{N, t})$ where
\begin{align*}
\mu_t = \sum_{i=0}^n t^i \mu_i, \quad d_{1,t} = \sum_{i=0}^n t^i d_{1, i}, \quad \ldots \quad , d_{N, t} = \sum_{i=0}^n t^i d_{N, i}
\end{align*}
such $(A[[t]] / (t^{n+1}), \mu_t, d_{1,t}, \ldots, d_{N, t})$ is an AssHDer pair over $\mathbb{K}[[t]]/ (t^{n+1})$.

Therefore, the following identities must hold: for $s = 0, 1, \ldots, n$,
\begin{align*}
\sum_{p+q = s} \mu_p (\mu_q (a, b), c) =~& \sum_{p+q = s} \mu_p (a, \mu_q (b, c)),\\
\sum_{p+q = s} d_{k, p} ( \mu_q (a, b)) =~& \sum_{i+j = k} \sum_{p+q+r = s} \mu_p ( d_{i, q}(a), d_{j, r} (b)), ~~~ \text{ for } k = 1, \ldots, N.
\end{align*}
In other words, for $s=0, 1, \ldots, n$,
\begin{align}
\frac{1}{2} \sum_{p+q=s, p,q > 0} [\mu_p, \mu_q] =~& - [\mu, \mu_s ] \label{n-def-eqn-1}\\
~~\mbox{and} ~~~ \sum_{p+q=s, p, q >0} [d_{k,p}, \mu_q ]_N =~& -[ d_k , \mu_s]_N + [\mu, d_{k,s}]_N, ~~ \text{ for } k =1, \ldots, N.  \label{n-def-eqn-2}
\end{align}

Let $(\mu_{n+1}; d_{1, n+1}, \ldots, d_{N, n+1}) \in C^2_{\mathrm{AssHDer}}(A, A)$ be a $2$-cochain such that $(\mu_t' = \mu_t + t^{n+1} \mu_{n+1}; d_{1, t}' = d_{1, t} + t^{n+1} d_{1, n+1}, \ldots, d_{N, t}' = d_{N, t} + t^{n+1} d_{N, n+1})$ defines a deformation of order $n+1$. Therefore, the above set of relations also holds for $s = n+1$,
\begin{align}
\sum_{p+q = n+1} \mu_p (\mu_q (a, b), c) =~& \sum_{p+q = n+1} \mu_p (a, \mu_q (b, c)), \label{ext-n+1}\\
\sum_{p+q = n+1} d_{k, p} ( \mu_q (a, b)) =~& \sum_{i+j = k} \sum_{p+q+r = n+1} \mu_p ( d_{i, q}(a), d_{j, r} (b)), ~~~ \text{ for } k = 1, \ldots, N. \label{extn+1}
\end{align}
In such a case, we say that the deformation $(\mu_t; d_{1,t}, \ldots, d_{N, t})$ is extensible. Note that the identities (\ref{ext-n+1}) and (\ref{extn+1}) can be written as
\begin{align*}
\delta_{\mathrm{Hoch}} (\mu_{n+1}) (a, b, c) =~& \sum_{p+q = n+1, p,q > 0} \mu_p ( \mu_q (a, b), c) - \mu_p (a, \mu_q (b, c)) \quad ( = \mathrm{Ob}(a,b,c) ~~~ \text{ say}) \\
(\delta_{\mathrm{Hoch}}' (d_{k, n+1}) + \delta_k \mu_{n+1}) (a, b) =~& \sum_{p+q = n+1, p, q >0} d_{k, p} ( \mu_q (a,b)) - \sum_{i+j = k} \sum_{p+q+r = n+1, p, q, r >0} \mu_p ( d_{i,q}(a), d_{j, r}(b)) \\ & \quad ( = \mathrm{Ob}_k (a,b) ~~~ \text{ say})
\end{align*}
for $k = 1, \ldots, N.$

\begin{prop}\label{3-thm}
The tuple $(\mathrm{Ob}; \mathrm{Ob}_1, \ldots, \mathrm{Ob}_N )$ is a $3$-cocycle in the cohomology of the AssHDer pair $A$ with coefficients in itself.
\end{prop}

\begin{proof}
It is known from the deformations of associative algebras \cite{gers} that $\mathrm{Ob}$ is a Hochschild $3$-cocycle, i.e. $\delta_{\mathrm{Hoch}} ( \mathrm{Ob}) = 0$. We also have
\begin{equation*}
\begin{split}
&\delta_{\mathrm{Hoch}}' (\mathrm{Ob}_k) + (-1)^3~ \delta_k (\mathrm{Ob}) \\
&= - [\mu, \mathrm{Ob}_k]_N + [d_k, \mathrm{Ob}]_N \\
&= - \sum_{i+j = n+1, i, j > 0} [\mu, [d_{k,i}, \mu_j ]_N ]_N + \frac{1}{2} \sum_{i+j = n+1, i, j \geq 0} [d_k, [\mu_i, \mu_j]]_N \\
&= - \sum_{i+j = n+1, i, j > 0} \big(  [[\mu, d_{k,i}]_N , \mu_j] + [d_{k,i}, [\mu, \mu_j]]_N    \big) + \frac{1}{2} \sum_{i+j = n+1, i, j > 0} \big(  [[d_k, \mu_i]_N, \mu_j] + [\mu_i, [d_k, \mu_j]_N ]   \big) \\
&= - \sum_{i+j = n+1, i, j > 0} \big(  [[\mu, d_{k,i}]_N , \mu_j] + [d_{k,i}, [\mu, \mu_j]]_N    \big) + \sum_{i+j = n+1, i, j > 0} [[d_k, \mu_i]_N , \mu_j] \\
&= - \sum_{\begin{array}{c}
{i+j = n+1},\\{ i, j > 0}\end{array}
} \big(  [[\mu, d_{k,i}]_N, \mu_j]      - [[d_k, \mu_i]_N, \mu_j]          ) + \frac{1}{2} \sum_{\begin{array}{c}
{i+ j' + j'' = n+1},\\ {i, j', j'' > 0}\end{array}
} [d_{k,i}, [\mu_{j'}, \mu_{j''}]]_N ~~  \quad (\mathrm{by }~~ (\ref{n-def-eqn-1}))\\
&= - \sum_{i'+ i'' + j = n+1, i', i'', j > 0} [[d_{k, i'}, \mu_{i''} ]_N , \mu_j ] - \sum_{i+j = n+1, i, j >0} ( [[d_k, \mu_i]_N, \mu_j] - [[ d_k , \mu_i]_N , \mu_j] ) \\
& \quad + \frac{1}{2} \sum_{i+ j' + j'' = n+1, i, j', j'' >0} ( [[d_{k,i}, \mu_{j'}]_N, \mu_{j''}] +[ \mu_{j'}, [ d_{k,i}, \mu_{j''} ]_N ] ) \qquad (\mathrm{by }~~ (\ref{n-def-eqn-2}))\\
&= - \sum_{\begin{array}{c}
{i' + i'' + j = n+1},\\{ i', i'', j > 0}\end{array}
} [[d_{k, i'}, \mu_{i''}]_N , \mu_j] + \sum_{\begin{array}{c}
{i+j' + j'' = n+1},\\ { i, j', j'' > 0}\end{array}
} [[ d_{k,i}, \mu_{j'}]_N, \mu_{j''}] = 0.
\end{split}
\end{equation*}

Thus 
\begin{align*}
\partial (\mathrm{Ob};~ \mathrm{Ob}_1, \ldots, \mathrm{Ob}_N ) = ( \delta_{\mathrm{Hoch}} \mathrm{Ob} ;~ \delta_{\mathrm{Hoch}}' ( \mathrm{Ob}_1 ) + (-1)^3 \delta_1 ( \mathrm{Ob}), \ldots, \delta_{\mathrm{Hoch}}' ( \mathrm{Ob}_N ) + (-1)^3 \delta_N ( \mathrm{Ob}) ) = 0.
\end{align*}
Hence the proof.
\end{proof}

Therefore we have $[(\mathrm{Ob}, \mathrm{Ob}_1, \ldots, \mathrm{Ob}_N )] \in H^3_{\mathrm{AssHDer}} (A, A)$. If this cohomology class is trivial, then there exists a $2$-cocycle $(\mu_{n+1}; d_{1, n+1}, \ldots, d_{N, n+1}) \in C^2_{\mathrm{AssHDer}} (A, A)$ such that
\begin{align*}
\partial (\mu_{n+1}; d_{1, n+1}, \ldots, d_{N, n+1}) = (\mathrm{Ob}; \mathrm{Ob}_1, \ldots, \mathrm{Ob}_N ).
\end{align*}
Then by the previous discussion $(\mu_t' = \mu_t + t^{n+1} \mu_{n+1}; d_{1,t}' = d_{1,t} + t^{n+1} d_{1, n+1}, \ldots, d_{N, t}' = d_{N, t} + t^{n+1} d_{N, n+1})$ is a deformation of order $n+1$. Thus, the deformation $(\mu_t; d_{1, t}, \ldots, d_{N, t})$ is extensible. Conversely, if the deformation $(\mu_t; d_{1, t}, \ldots, d_{N, t})$ is extensible, then the obstruction cocycle is given by the coboundary $\partial (\mu_{n+1}; d_{1, n+1}, \ldots, d_{N, n+1})$. Hence the corresponding cohomology class is null.

\begin{thm}
Let $(\mu_t; d_{1, t}, \ldots, d_{N, t})$ be a deformation of order $n$ of the AssHDer pair $(A, d_{1}, \ldots, d_N)$. It is extensible if and only if the corresponding obstruction class $[(\mathrm{Ob}; \mathrm{Ob}_1, \ldots, \mathrm{Ob}_N )]$ is null.
\end{thm}

\begin{thm}\label{final-thm}
If $H^3_{\mathrm{AssHDer}} (A, A) = 0$ then every finite order deformation of $A$ is extensible.
\end{thm}

\begin{corollary}
If $H^3_{\mathrm{AssHDer}} (A, A) = 0$ then every $2$-cocycle in the cohomology of $A$ is the infinitesimal of a formal deformation of $A$.
\end{corollary}

\section{Conclusions}
In this paper, we study associative algebras equipped with higher derivations. We call such a tuple of an associative algebra and a higher derivation by AssHDer pair. We define a suitable cohomology for an AssHDer pair with coefficients in a representation. This cohomology is related to extensions of AssHDer pair and govern deformations of the structure.

Deformations of associative algebras can be generalized to operads with multiplications \cite{das-loday}. Various Loday-type algebras (e.g. dendriform algebras, tridendriform algebras, associative dialgebras, quadri algebras) can be defined in terms of operad with multiplications. Therefore, deformations of such algebras are particular cases. See \cite{yau, das-loday} for cohomology of these Loday-type algebras. Given an operad $\mathcal{O}$ with a multiplication $m \in \mathcal{O}(2)$ (i.e. $m$ satisfies $m \circ_1 m =  m \circ_2 m)$, a sequence $\{d_1, \ldots, d_N \}$ of elements in $\mathcal{O}(1)$ is called a higher derivation of rank $N$ for $m$ if they satisfy
\begin{align*}
d_k \circ m = \sum_{i+j = k} m \circ (d_i , d_j ), ~~~ \text{ for } k =1, \ldots, N.
\end{align*}
The method of the present paper can be generalized to study deformations of a tuple $(m, d_1, \ldots, d_N)$. Therefore, one may deduce deformations of Loday-type algebras equipped with higher derivations.

Deformations of algebras over a quadratic operad $\mathcal{P}$ are studied in \cite{bala}. In \cite{loday} the author defines derivation in a $\mathcal{P}$-algebra. In a subsequent paper, we aim to define higher derivations on $\mathcal{P}$-algebras and construct explicit cohomology and deformation theory of $\mathcal{P}$-algebras with higher derivations.

\begin{remark}
Like derivations are generalization of usual derivative, Rota-Baxter operators on associative algebras are generalization of the integral operator. Let $A$ be an associative algebra. A linear map $R : A \rightarrow A$ is called a Rota-Baxter operator (of weight $0$) on $A$ if $R$ satisfies
\begin{align*}
R(a ) R(b) = R (a R(b) + R(a) b), ~~~ \text{ for } a, b \in A.
\end{align*}
Free Rota-Baxter algebras over a vector space was defined in \cite{fard-guo} using rooted trees. Deformations of Rota-Baxter algebras are studied in \cite{das}. See \cite{guo-book} for more details about Rota-Baxter algebras. By generalizing higher integral operators,
one may define the notion of higher Rota-Baxter operator of rank $N$. More precisely, a higher Rota-Baxter operator of rank $N$ consists of a sequence of linear maps $\{R_1, \ldots, R_N \}$ satisfying the following identities
\begin{align*}
R_k (a) R_k (b ) = R_k ( \sum_{i+j = k} R_i (a) R_j (b) ), ~~ \text{ for } k = 1, \ldots, N
\end{align*}
with the convention that $R_0 = \mathrm{id}_A$. If $R$ is a Rota-Baxter operator on $A$ then it is easy to see that $\{ R, \frac{R^2}{2!}, \ldots, \frac{R^N}{N!} \}$ is a higher Rota-Baxter operator of rank $N$.

In \cite{guo-keigher} the authors also consider algebras with both differential operators and integral operators. More precisely, a differential Rota-Baxter algebra (of weight $0$) is an algebra $A$ with a derivation $d$ and a Rota-Baxter operator $R$ such that $d \circ R = \mathrm{id}_A$. In a similar way, we may define higher differential Rota-Baxter algebras. A higher differential Rota-Baxter algebra of rank $N$ is an algebra $A$ with a higher derivation $\{ d_1, \ldots, d_N \}$ and a higher Rota-Baxter operator $\{R_1, \ldots, R_N \}$ satisfying $d_k \circ R_k = \frac{1}{(k!)^2} \mathrm{id}_A$. In future, we come with more structural properties of higher Rota-Baxter operators, higher differential Rota-Baxter algebras and the corresponding free objects.

\end{remark}

\noindent {\em Acknowledgements.} The author is thankful to Indian Institute of Technology Kanpur for financial support.

\end{document}